\documentclass[11pt]{amsart}

\usepackage{amsmath,amssymb,latexsym,soul,cite,mathrsfs}

\usepackage{color,enumitem,graphicx}
\usepackage[colorlinks=true,urlcolor=blue,
citecolor=red,linkcolor=blue,linktocpage,pdfpagelabels,
bookmarksnumbered,bookmarksopen]{hyperref}
\usepackage[english]{babel}

\usepackage[left=2.9cm,right=2.9cm,top=2.8cm,bottom=2.8cm]{geometry}
\usepackage[hyperpageref]{backref}

\usepackage[colorinlistoftodos]{todonotes}
\makeatletter
\providecommand\@dotsep{5}
\def\listtodoname{List of Todos}
\def\listoftodos{\@starttoc{tdo}\listtodoname}
\makeatother

\numberwithin{equation}{section}

\newtheorem{theorem}{Theorem}[section]
\newtheorem{proposition}[theorem]{Proposition}
\newtheorem{lemma}[theorem]{Lemma}
\newtheorem{corollary}[theorem]{Corollary}

\newtheorem{claim}[theorem]{Claim}

\newtheorem{remark}{Remark}
\newtheorem{definition}[theorem]{Definition}

\begin{document}
	
	\title[Existence of solution for Schr\"odinger equation with discontinuous nonlinearity]
	{Existence of solution for Schr\"odinger equation with discontinuous nonlinearity and critical Growth}

	\author{Geovany F. Patricio}

	\address[Geovany F. Patricio]
	{\newline\indent Unidade Acad\^emica de Matem\'atica
		\newline\indent 
		Universidade Federal de Campina Grande,
		\newline\indent
		58429-970, Campina Grande - PB - Brazil}
	\email{\href{fernandes.geovany@yahoo.com.br}{fernandes.geovany@yahoo.com.br}}

	\pretolerance10000
	
	
	\begin{abstract}
		\noindent This paper concerns with the existence of nontrivial solution for the following problem
		\begin{equation}
		\left\{\begin{aligned}
		-\Delta u + V(x)u & = \gamma H_{e}(|u|-a)|u|^{q-2}u+|u|^{2^{*}-2}u\;\;\mbox{ in}\;\;\mathbb{R}^{N},\nonumber \\
		u \in H^{1}(\mathbb{R}^{N}). 
		\end{aligned}
		\right.
		\end{equation} 
where, $N\geq 3$, $\gamma\geq 0$, $H_{e}:\mathbb{R} \to \mathbb{R}$ denotes the Heaviside function,
$a \geq 0$, $2<q<2^*$ and  $V:\mathbb{R^{N}} \to \mathbb{R}$ is $\mathbb{Z}^{N}$-periodic with $\beta= 0$ does not belong to the spectrum of $-\Delta+V$.
\end{abstract}

	\thanks{Geovany F. Patricio was
		supported by  CAPES, Brazil }


		\maketitle
\section{Introduction}

In this paper we study the existence of nontrivial solution for the following class of elliptic problems
$$
\left\{\begin{aligned}
-\Delta u + V(x)u &= f_{\gamma}(x,u)\;\;\mbox{in}\;\;\mathbb{R}^{N} \\
u \in H^{1}(\mathbb{R}^{N}),
\end{aligned}
\right. \leqno{(P_{\gamma}^{*})}
$$
where, $f_{\gamma}(x,s)=\gamma H_{e}(|s|-a)|s|^{q-2}s+|s|^{2^{*}-2}s$
and $H_e:\mathbb{R} \to \mathbb{R}$ denotes the Heaviside function, that is, 
$$
H_e(t)=
\left\{
\begin{array}{l}
	0, \quad \mbox{if} \quad t \leq 0\\
	1, \quad \mbox{if} \quad t > 0.
\end{array}
\right.
$$ 
In addition, $\gamma\geq 0$, $ N \geq 3$ and $V$ is continuou function, periodic with respect to $x$-variable,
$$
0 \notin \sigma(-\Delta+V),\;\;\mbox{the spectrum of}\;\; -\Delta +V; \leqno{(V_{1})}
$$  
and
$$
 \sigma(-\Delta+V)\cap (-\infty, 0)\neq \emptyset. \leqno{(V_{2})}
$$  

Problems of type $(P_{\gamma}^{*})$, when nonlinearity is considered continuous, has been studied by Chabrowski and Szulkin \cite{Chabrowski}, M. Schechter and W. Zou \cite{Schechter}, Alves and Germano \cite{AG1}. In \cite{Chabrowski} Chabrowski and Szulkin studied the following class of problems
$$
\left\{\begin{aligned}
-\Delta u + V(x)u &= K(x) |u|^{2^{*}-2}u+g(x,u) \;\;\mbox{in}\;\;\mathbb{R}^{N} \\
u \in H^{1}(\mathbb{R}^{N}),
\end{aligned} 
\right.\leqno{(P)}
$$
where $N\geq 4$, $V$, $K$, $g$ are periodic in $x_{j}$ for $1 \leq  j \leq N$, $K > 0$, $g$ is of subcritical growth and $0$ is in a gap of the spectrum of $-\Delta+V$. They proved the existence of a nontrivial solution using the linking theorem. 

Knowing that the equation
\begin{equation}\label{Crit}
	-\Delta u+\beta u= |u|^{2^{*}-2}u,\;\mbox{in}\;\;\mathbb{R}^{N},
\end{equation}
when $\beta \neq 0$, has only the trivial solution in $H^{1}(\mathbb{R}^{N})$(cf \cite{Benci}). Thereby, when $\gamma=0$ in $(P_{\gamma}^{*})$ we are left with an equation similar to (\ref{Crit}) and so the existence of nontrivial solution of $(P_{\gamma}^{*})$ is an interesting problem. 
One of the pioneering results involving critical problems was obtained in article of Br\'ezis and Nirenberg \cite{Nirenberg} and has helped and motivated a great amount of research on this class of problem involving critical exponent. 

In several articles, the authors are focus on equations with subcritical growth, since critical growth bring many difficulties, because, yet in bounded domain $\Omega \subset \mathbb{R}^{N}$, the Sobolev embedding  $H_{0}^{1}(\Omega)\hookrightarrow L^{2^{*}}(\Omega)$ is not compact. In recent years, the study of equations with critical growth has made great progress and has attracted the attention of many authors. One way to regain compactness is constrain the functional value in a suitable interval, see for example \cite{Chabrowski}. When the functional is strongly defined, then the functional has mountain pass geometry, so it is easy to estimate the functional level. For example, under some weaker conditions, Lins and Silva \cite{Lins} considered the existence of nontrivial solutions 
\begin{equation*}
-\Delta u+V(x)u=f(x,u)
\end{equation*}
 when $f$ is asymptotically periodic in $x$. However, for the strongly indefinite case, the problem becomes very difficult e.g. see \cite{Chabrowski, Schechter, AG1}.

We will study a class of semilinear problems with discontinuous linearity in $\mathbb{R}^{N}$, strongly indefinite and critical growth $(P_{\gamma}^{*})$. The interest in the study of nonlinear partial differential equations with discontinuous nonlinearities has increased because many free boundary problems arising in mathematical physics may be stated in this from. Among these problems, we have the seepage surface problem and the Elenbaas equation, see for example \cite{Chang1,chang2, chang3}.

Variational methods for problems with discontinuous nonlinearity and critical exponent have been applied to several problems.  Alves and Bertone \cite{alves3}, got two nonnegative solutions for the following quasilinear problem
\begin{equation*}
-\Delta_{p}u = H_{e}(u-a)u^{p^{*}-1}+\lambda h(x),\;\mbox{in}\;\mathbb{R}^{N}
\end{equation*}
where $H_{e}$ is the Heaviside function. Alves, Bertone and Goncalves in \cite{alves1} studied the equation
\begin{equation*}
-\Delta u= u^{2^{*}-1}+\lambda h(x)H(u-a)u^{q},\;\mbox{in}\;\mathbb{R}^{N}
\end{equation*}
obtained the existence of positive solutions for $0\leq q\leq 1$ and $1<q<2^{*}-1.$

A rich literature is available for problems with discontinuous nonlinearities, and we refer the reader to Chang \cite{Chang1}, Ambrosetti and Badiale \cite{Badiale}, Alves and Patricio \cite{AlvesPatricio}, Cerami \cite{cerami}, Alves et al. \cite{alves1}, Alves et al. \cite{alves2}, Alves and Bertone \cite{alves3}, Alves and Nascimento \cite{alves4}, Cerami \cite{cerami, Cerami}, Badiale \cite{badiale2}, Dinu \cite{dinu}, Gasi\'nski and Papageorgiou \cite{GP},  Kourogenis and Papageorgiou \cite{Papageorgiou}, Mironescu and R\u adulescu \cite{Radulescu 2}, R\u adulescu \cite{Radulescu, Radulescu 1, Radulescu 4, Radulescu 5, Radulescu 6, Radulesco 7} and their references. Several techniques have been developed or applied in their study, such as variational methods for nondifferentiable functionals, lower and upper solutions, global branching, and the theory of multivalued mappings.

In this paper a study is made of a strongly indefinite problem with discontinuous nonlinearity that involves two extreme conditions that result in loss of compactness, namely, when we consider the space $\mathbb{R}^{N}$ and the critical growth. More precisely, we will find nontrivial solution to problem $(P_{\gamma}^{*})$ and for this we will use a link theorem for a class of functional locally Lipschitz due to Alves and Patricio \cite{AG}, in which they propose a generalization of the link terorema due to Kryszewski and Szulkin \cite{Kryszewski} in which they consider class $C^{1}$ functional. One of the main difficulty in the present paper was to make sure that the weak limit of sequences $(PS)$ is nontrivial for a given level. 

The present paper is relevant, because it is the first result that establishes the existence of solutions for Schr\"odinger equation strongly indefinite with discontinuous nonlinearity and critical growth.

The main result is the following
\begin{theorem}\label{Teorema1}
Suppose that conditions $(V_{1})-(V_{2})$ are satisfied. If $N \geq 4$, the problem $(P_{\gamma}^{*})$ has a nontrivial solution. If $N=3$, there are $\gamma>0$ and $a_{0}=a_{0}(\gamma)>0$ such that, $0\leq a \leq a_{0}$, the problem $(P_{\gamma}^{*})$ has a nontrivial solution.
\end{theorem}
\begin{definition}
For a solution to the problem  $(P_{\gamma}^{*})$ 
we understand it to be a function $u \in H^{1}(\mathbb{R^{N}})\cap W_{loc}^{2,p}(\mathbb{R^{N}})$, for some $p>1$, such that
$$
-\Delta u(x) + V(x) u(x) \in \partial_{t} F_{\gamma}(x,u(x))\;\;\mbox{a.e in}\;\; \mathbb{R}^{N},
$$
where $F_{\gamma}(x,t)= \int_{0}^{t}f_{\gamma}(x,s)ds$ and $\partial_t F_{\gamma}$ denotes the generalized gradient of $F_{\gamma}$ with respect to variable $t$. 
When	
\begin{equation*}
-\Delta u(x)+V(x)u(x)=f_{\gamma}(x,u(x)),\;\mbox{a.e in}\;\mathbb{R}^{N},
\end{equation*}	
then we say that $u$ is strong solution.
\end{definition}
\begin{remark}\label{Obs1}
In case $N\geq 4$, we will not have restrictions for constants $a\geq 0$ and $\gamma\geq 0$ to obtain a nontrivial solution to the problem $(P_{\gamma}^{*})$. Thereby, in case $N\geq 4$, we will have a strong solution for $(P_{\gamma}^{*})$ 
choosing $a>0$ so that
\begin{equation}\label{V}
\sup_{x \in \mathbb{R}^{N}}V(x)<a^{2^{*}-2}.
\end{equation}
First note that (\ref{V}) is possible because we are assuming that $V:\mathbb{R}^{N}\rightarrow \mathbb{R}$ is $\mathbb{Z}^{N}$-periodic and therefore bounded, i.e, there exists $K_{0}>0$ such that $|V(x)|\leq K_{0}$ for all $x \in \mathbb{R}^{N}$. So, just choose $a>K_{0}$.

Using the definition of generalized gradient, by a simple calculation, we get:
$$
\partial_{t}F_{\gamma}(x,u(x))= \left\{\begin{aligned}
u(x)^{2^{*}-1},\;\;\mbox{if}\;\; 0\leq u(x)<a \\
[a^{2^{*}-1}, a^{2^{*}-1}+\gamma a^{q-1}],\;\; \mbox{if}\;\; u(x)=a \\
u(x)^{2^{*}-1}+\gamma u(x)^{q-1},\;\; \mbox{if}\;\; u(x)>a.
\end{aligned}
\right. 
$$
or
$$
\partial_{t}F_{\gamma}(x,u(x))= \left\{\begin{aligned}
(-u(x))^{2^{*}-2}u(x),\;\;\mbox{if}\;\; -a<u(x)\leq 0 \\
[ (-a)^{2^{*}-1}+\gamma (-a)^{q-1}, (-a)^{2^{*}-1}],\;\; \mbox{if}\;\; u(x)=-a \\
(-u(x))^{2^{*}-2}u(x)+\gamma (-u(x))^{q-2}u(x),\;\; \mbox{if}\;\; u(x)<-a.
\end{aligned}
\right. 
$$	
\end{remark}
Suposse that $u$ is a solution of $(P_{\gamma}^{*})$, then
$$
 \left\{\begin{aligned}
-\Delta u(x)+V(x)u(x)=u(x)^{2^{*}-1},\;\;\mbox{if}\;\; 0\leq u(x)<a \\
-\Delta u(x)+V(x)u(x) \in [a^{2^{*}-1}, a^{2^{*}-1}+\gamma a^{q-1}],\;\; \mbox{if}\;\; u(x)=a \\
-\Delta u(x)+V(x)u(x)=u(x)^{2^{*}-1}+\gamma u(x)^{q-1},\;\; \mbox{if}\;\; u(x)>a.
\end{aligned}
\right. 
$$
or
$$
\left\{\begin{aligned}
-\Delta u(x)+V(x)u(x)=(-u(x))^{2^{*}-2}u(x),\;\;\mbox{if}\;\; -a<u(x)\leq 0 \\
-\Delta u(x)+V(x)u(x) \in [ (-a)^{2^{*}-1}+\gamma (-a)^{q-1}, (-a)^{2^{*}-1}],\;\; \mbox{if}\;\; u(x)=-a \\
-\Delta u(x)+V(x)u(x)=(-u(x))^{2^{*}-2}u(x)+\gamma (-u(x))^{q-2}u(x),\;\; \mbox{if}\;\; u(x)<-a.
\end{aligned}
\right. 
$$
Set 
\begin{equation*}
A_{+}=\{x \in \mathbb{R}^{N}: u(x)=a\}\;\;\mbox{and}\;\;A_{-}=\{x \in \mathbb{R}^{N}: u(x)=-a\}.
\end{equation*}
Under (\ref{V}) we obtain that the Lebesgue
measure of $A_{+}$ (or $A_{-}$) is zero. Otherwise if measure of $A_{+}$ (or  $A_{-}$) is positive, then by applying Stampacchia theorem \cite{Stampacchia},
\begin{equation*}
-\Delta u(x)=0\;\;\mbox{in}\;\;  A_{+}\; (\mbox{or}\; A_{-}),
\end{equation*}
that is, 
\begin{eqnarray}
&&V(x)a \in [a^{2^{*}-1}, a^{2^{*}-1}+\gamma a^{q-1}],\;\mbox{for}\;\; x \in A_{+}, \nonumber \\
\mbox{or}\nonumber \\
&&-V(x)a \in [ (-a)^{2^{*}-1}+\gamma (-a)^{q-1}, (-a)^{2^{*}-1}],\;\mbox{for}\;\; x \in A_{-}. \nonumber
\end{eqnarray}
In any of the above cases, once $(-a)^{2^{*}-2}= a^{2^{*}-2}$, we have
\begin{equation*}
a^{2^{*}-2}\leq V(x),\;\mbox{for all}\;\;x \in A_{+}\;(\mbox{or}\;A_{-}),
\end{equation*}
a contradiction with (\ref{V}).

By remark \ref{Obs1}, we can enunciate the following:
\begin{corollary}
	If $N\geq 4$, assuming $(V_{1})$ and $(V_{2})$, there exist $a>0$ such that the problem $(P_{\gamma}^{*})$ has a nontrivial strong solution.
\end{corollary}

\noindent \textbf{Notation:} From now on, otherwise mentioned, we use the following notations:
\begin{itemize}
	\item $B_r(u)$ is an open ball centered at $u$ with radius $r>0$, $B_r=B_r(0)$.
	
	\item $X^{*}$ denotes the dual topological space of $X$ and $||\;\;\;||_{*}$ denotes the norm in $X^{*}$.
	
	\item   $C$ denotes any positive constant, whose value is not relevant.
	
	\item  $||\,\,\,||_p$ denotes the usual norm of the Lebesgue space $L^{p}(\mathbb{R}^N)$, for $p \in [1,+\infty]$.
	
	\item  $||\,\,\,||_{H}$ denotes the usual norm of the Orlicz space $L^{H}(\mathbb{R}^N)$ associated the $N$-function $H$.
	
	\item If $u:\mathbb{R}^N \to \mathbb{R}$ is mensurable function, the integral $\int_{\mathbb{R}^N}u\,dx$ will be denoted by $\int_{\mathbb{R}^N}u$.
	
	\item $o_{n}(1)$ denotes a real sequence with $o_{n}(1)\to 0$ as $n \to +\infty$.
	
	\item If $k>0$, $O(\varepsilon^{k})$ denotes a function that $\frac{O(\varepsilon^{k})}{\varepsilon^{k}}$ is bounded as $\varepsilon\rightarrow 0^{+}$.
	
\end{itemize}

\section{Basic results from nonsmooth analysis}

In this section, for the reader's convenience, we recall some definitions and basic results on the critical point theory of locally Lipschitz functionals as developed by Chang \cite{Chang1}, Clarke \cite{Clarke, Clarke1} and Grossinho and Tersian \cite{rosario}.

Let $(X, ||\;||_{X})$ and $(Y, ||\;||_{Y})$ be a real Banach spaces. A functional $I :X \rightarrow \mathbb{R} $ is locally Lipschitz, $I\in Lip_{loc}(X, \mathbb{R})$ for short, if given $u \in X$ there is an open neighborhood $V:=V_{u} \subset X$ of $u$, and a constant $K=K_{u}>0$ such that
\begin{equation*}
	|I(v_{2})-I(v_{1})| \leq K ||v_{1}-v_{2}||_{X},\;\;v_{i} \in V,\;i=1,2.
\end{equation*}
The generalized directional derivative of $I$ at $u$ in the direction of $v \in X$ is defined by
\begin{equation*}
	I^{\circ}(u;v)= \limsup_{h\rightarrow 0, \delta \downarrow 0}\frac{1}{\delta}\left(I(u+h+\delta v)-I(u+h)\right).
\end{equation*}
The generalized gradient of $I$ at $u$ is the set
\begin{equation*}
	\partial I(u)= \{\xi \in X^{*}\;;\; I^{\circ}(u;v) \geq \left<\xi, v\right>\,;\forall\; v \in X \}.
\end{equation*} 
\begin{lemma}
	If $I$ is continuously differentiable to Fr\'echet in an open neighborhood of $u \in X$, we have $\partial I(u)= \{I'(u)\}$.
\end{lemma}
\begin{lemma}\label{35}
	If $Q \in C^{1}(X, \mathbb{R})$ and $\Psi \in Lip_{loc}(X, \mathbb{R})$, then for each $u \in X$
	$$
	\partial(Q+\Psi)(u)= Q'(u) + \partial \Psi(u).
	$$
\end{lemma}
Moreover, we denote by $\lambda_{I}(u)$ the following real number 
$$
\lambda_{I}(u):=\min\{||\xi||_{*}: \xi \in \partial I(u)\}.
$$
We recall that $u \in X$ is a critical point of $I$ if $0 \in \partial I(u)$, or equivalently, when $\lambda_I(u)=0$. 

%

\section{On the energy functional of problem $(P_{\gamma}^{*})$}

It follows  that the functional
$$I_{\gamma}(u)= \frac{1}{2}\int_{\mathbb{R}^{N}}(|\nabla u|^{2}+V(x)u^{2})-\int_{\mathbb{R}^{N}}F_{\gamma}(x,u), \;u \in H^{1}(\mathbb{R}^{N}),$$
where $F_{\gamma}(x,t)= \int_{0}^{t}f_{\gamma}(x,s)ds$, is well defined. 

By standard argument, $Q \in C^{1}(H^{1}(\mathbb{R}^{N}), \mathbb{R})$
where
$$ Q(u)= \frac{1}{2}\int_{\mathbb{R}^{N}}(|\nabla u|^{2}+V(x)u^{2})$$
and 
$$Q'(u)v=\int_{\mathbb{R}^{N}}(\nabla u \nabla v+V(x)u v),\;\forall\; u,v \in H^{1}(\mathbb{R}^{N}).$$
In addition, we can write  
$$I_{\gamma}= Q- \Psi_{\gamma},$$
where $\Psi_{\gamma}: H^{1}(\mathbb{R}^{N})\rightarrow \mathbb{R}$ given by
\begin{equation*} 
\Psi_{\gamma}(u)=\int_{\mathbb{R}^{N}}F_{\gamma}(x,u).
\end{equation*}
By $(V_{1})$, it is well known that $H^{1}(\mathbb{R}^{N})=X^{+}\oplus X^{-}$ is a orthogonal decomposition and there is an equivalent norm $||\cdot||$ to $||\cdot||_{H^{1}(\mathbb{R}^{N})}$ (see \cite{AG, AlvesPatricio}) such that
\begin{equation}\label{Functional}
I_{\gamma}(u)= \frac{1}{2}||u^{+}||^{2}-\frac{1}{2}||u^{-}||^{2}-\Psi_{\gamma}(u),\;\forall\; u=u^{+}+u^{-} \in X^{+} \oplus X^{-} .
\end{equation}
By definition
\begin{equation}\label{Grad.}
\partial_{t} F_{\gamma}(x,t)= \{\mu \in \mathbb{R}\;:\; F_{\gamma}^{\circ}(x,t;r) \geq \mu r,\; r \in \mathbb{R}\},
\end{equation}
where $F_{\gamma}^{\circ}(x,t;r)$ denotes the generalized directional derivative of $t\mapsto F_{\gamma}(x,t)$ in the direction
of $r$, i.e,
\begin{equation*}
F_{\gamma}^{\circ}(x,t;r)= \limsup_{h\rightarrow t, \lambda \downarrow 0} \frac{F_{\gamma}(x, h+\lambda r)-F_{\gamma}(x,h)}{\lambda}.
\end{equation*}
Consider
\begin{eqnarray}
F_{\gamma}:\mathbb{R}^{N}&\times& \mathbb{R}\longrightarrow \mathbb{R} \nonumber \\
&(x,t)&\longmapsto F_{\gamma}(x,t) = \int_{0}^{t}f_{\gamma}(x,s) ds \nonumber
\end{eqnarray}
where $f_{\gamma}(x,s)= \gamma H(|s|-a)|s|^{q-2}s+|s|^{2^{*}-2}s$, that is, 
$$
F_{\gamma}(x,t)= \left\{\begin{aligned}
\frac{1}{2^{*}}|t|^{2^{*}},\;\;\mbox{if}\;\; |t|\leq a \\
\frac{1}{2^{*}}|t|^{2^{*}}+\frac{\gamma}{q}|t|^{q}- \frac{\gamma}{q}a^{q},\;\; \mbox{if}\;\; |t|>a
\end{aligned}
\right. 
$$
In fact it is easy to check that functional $\Psi_{\gamma}: H^{1}(\mathbb{R}^{N})\rightarrow \mathbb{R}$ given by
\begin{equation} \label{Psi}
\Psi_{\gamma}(u)=\int_{\mathbb{R}^{N}}F_{\gamma}(x,u),
\end{equation}
is well defined. However, in order to apply variational methods it is better to consider the functional $\Psi$ in a more appropriated domain, that is, $\Psi_{\gamma}: L^{\Phi}(\mathbb{R}^{N})\rightarrow \mathbb{R}$, for $\Phi(t)=|t|^{q}+|t|^{2^{*}}$, where $ L^{\Phi}(\mathbb{R}^{N})$ denotes the Orlicz space associated with the $N$-function $\Phi$ (for more details on Orlicz space see \cite{Fukagai 1, RAO, T.K}). 

Since $\Phi$ satisfies $\Delta_{2}$-condition, we can guarantee that given $J \in (L^{\Phi}(\mathbb{R}^{N}))^{*}$, then
$$J(u)= \int_{\mathbb{R}^{N}} vu\;,\;\;\forall\; u \in L^{\Phi}(\mathbb{R}^{N}),$$
for some $v \in L^{\tilde{\Phi}}(\mathbb{R}^{N})$, where $\tilde{\Phi}$ is the conjugate function of $\Phi$. In general, we need to prove that the inclusion below holds
$$\partial \Psi_{\gamma} (u)\subset \partial_{t} F_{\gamma}(x,u)= [\underline{f_{\gamma}}(x, u(x)), \overline{f_{\gamma}}(x, u(x))]\;\;\mbox{a.e in}\;\; \mathbb{R}^{N},$$
where
$$\underline{f_{\gamma}}(x,t)= \lim_{r\downarrow 0} ess \inf\{f_{\gamma}(x,s); |s-t|<r\}$$
and
$$\overline{f_{\gamma}}(x,t)= \lim_{r\downarrow 0} ess \sup\{f_{\gamma}(x,s); |s-t|<r\}.$$
We have that the condition below is satisfied:
\begin{itemize}
	\item [($F_{*}$)]There exist $C_0, C_1>0$ such that 
	$$|\xi| \leq C_0(|u|^{q-1} +|u|^{2^{*}-1}) \leq C_1\Phi'(|u|), \;\;\forall\; \xi \in \partial_{t} F_{\gamma}(x,u),\;\forall\; x \in \mathbb{R}^{N},$$
	for some $C_1>0$.
\end{itemize}

The next three results establish important properties of the functional $\Psi_{\gamma}$ given in (\ref{Psi}).   

\begin{lemma}\label{Lip}[See \cite{Alves1} or \cite{AG1}]
	Assume ($F_{*}$). Then, the functional $\Psi_{\gamma}: L^{\Phi}(\mathbb{R}^{N})\rightarrow \mathbb{R}$ given by
	$$\Psi_{\gamma}(u)= \int_{\mathbb{R}^{N}} F_{\gamma}(x,u),\; u \in L^{\Phi}(\mathbb{R}^{N}),$$
	is well defined and $\Psi_{\gamma} \in Lip_{loc}(L^{\Phi}(\mathbb{R}^{N}), \mathbb{R})$.
\end{lemma}

\begin{theorem}[See \cite{Chang1}, Theorem 2.1  or \cite{Alves}, Theorem 4.1]\label{72}
	Assume ($F_{*}$), then for each $u \in L^{\Phi}(\mathbb{R}^{N})$,
	\begin{equation}\label{67}
	\partial \Psi_{\gamma} (u) \subset \partial_{t} F_{\gamma}(x, u) = [\underline{f_{\gamma}}(x,u(x)), \overline{f_{\gamma}}(x,u(x))]\;\;\mbox{a.e in}\;\; \mathbb{R}^{N}.
	\end{equation}
\end{theorem}
The inclusion above means that given $\xi \in  \partial \Psi_{\gamma} (u) \subset (L^{\Phi}(\mathbb{R}^{N}))^{*}\approx L^{\tilde{\Phi}}(\mathbb{R}^{N})$, there is $\tilde{\xi} \in L^{\tilde{\Phi}}(\mathbb{R}^{N})$ such that  
\begin{itemize}
	\item $\left<\xi, v \right>= \int_{\mathbb{R}^{N}} \tilde{\xi}v ,\;\forall\; v \in L^{\Phi}(\mathbb{R}^{N})$,
	\item $\tilde{\xi}(x) \in \partial_{t} F_{\gamma}(x, u(x)) = [\underline{f_{\gamma}}(x,u(x)), \overline{f_{\gamma}}(x,u(x))]\;\;\mbox{a.e in}\;\; \mathbb{R}^{N}.$
\end{itemize}
The following proposition is very important to establish the existence of a critical point for the functional $I_{\gamma}$.
\begin{proposition}(See \cite{AG}).\label{7}
	If $(u_{n}) \subset H^{1}(\mathbb{R}^{N})$  is such that 
	$u_{n}\rightharpoonup u_{0}$ in $H^{1}(\mathbb{R}^{N})$ and $\rho_{n} \in \partial \Psi_{\gamma}(u_{n})$ satisfies $\rho_{n} \stackrel{\ast}{\rightharpoonup} \rho_{0}$ in $(H^{1}(\mathbb{R}^{N}))^{*}$, then $\rho_{0} \in \partial \Psi_{\gamma}(u_{0})$. 
\end{proposition}

\section{Generalized linking theorem}	

	From now on, $X$ is a Hilbert space with $X=Y \oplus Z$, where $Y$ is a separable closed subspace of $X$ and $Z=Y^{\perp}$. If $u \in X$,  $u^{+}$ and $u^{-}$ denote the orthogonal projections from $X$ in $Z$ and in $Y$, respectively. In $X$ let us define the norm
	\begin{eqnarray}
	|||\cdot|||: &X &\longrightarrow \mathbb{R} \nonumber \\
	&u& \longmapsto |||u|||=\max\left\{||u^{+}||, \sum_{k=1}^{\infty}\frac{1}{2^{k}}|(u^{-},e_{k})|\right\}, \nonumber
	\end{eqnarray}
	where $(e_{k})$ is a total orthonormal sequence in $Y$. The topology on $X$ generated by $|||\cdot|||$ will be denoted by $\tau$ and all topological notions related to it will include this symbol.
		
Let $I :X \rightarrow \mathbb{R}$ a funtional locally Lipschitz, $I\in Lip_{loc}(X, \mathbb{R})$. We will say a functional $I:X\rightarrow \mathbb{R}$ verifies the condition $(H)$ when: \\
$$
(H) \quad \left\{
\begin{array}{l}
	\mbox{If $(u_{n}) \subset I^{-1}([\alpha, \beta])$ is such that $u_{n}\stackrel{\tau}{\rightarrow} u_{0}$ in $X$, then there exists $M>0$ }  \\
	\mbox{such that $\partial I(u_{n}) \subset B_{M}(0) \subset X^{*},\;\forall\; n \in \mathbb{N}$. In addition, if  $\xi_{n} \in \partial I(u_{n})$} \\
	\mbox{ with $\xi_{n} \stackrel{\ast}{\rightharpoonup} \xi_{0}$ in  $X^{*}$, we have $\xi_{0} \in \partial I(u_{0})$.}
\end{array}
\right\}
$$
\begin{theorem}\label{Link}(See \cite{AG}).
Let $Y$ be a separable closed subspace of a Hilbert space $X$ and $Z=Y^{\perp}$. If $u \in X$, as in the previous section, $u^{+}$ and $u^{-}$ denote the orthogonal projections in $Z$ and $Y$, respectively. 

Given $\rho >r>0$ and $z \in Z$ with $||z||=1$, we set 
\begin{eqnarray}
&&\mathcal{M}=\left\{u=y+t z\;; ||u||\leq \rho, t \geq 0 \;\;\mbox{and}\;\; y \in Y\right\} \nonumber \\
&& \mathcal{M}_{0}=\left\{u=y+t z\;;y \in Y, ||u||= \rho\;\;\mbox{and}\;\; t \geq 0 \;\;\mbox{or}\;\;  ||u||\leq  \rho\;\;\mbox{and}\;\; t=0 \right\} \nonumber \\
&& S=\left\{u \in Z\;; ||u||=r\right\}. \nonumber
\end{eqnarray} 
Assume $I \in Lip_{loc}(X, \mathbb{R})$ such that
\begin{equation*}
I \;\mbox{is}\;\;\tau-\mbox{upper semicontinuous}
\end{equation*}
and
\begin{equation*}
b=\inf_{S}I> \sup_{\mathcal{M}_{0}}I\;,\; d=\sup_{\mathcal{M}} I < \infty.
\end{equation*}
If $I$ verifies the condition $(H)$, there is $c \in [b,d]$ and a sequence $(u_{n}) \subset X$ such that
$$I(u_{n}) \rightarrow c\;\;\mbox{and}\;\; \lambda_{I}(u_{n})\rightarrow 0.$$	
\end{theorem}

\section{Proof of Theorem \ref{Teorema1}}
	
It follows, by Alves and Patricio \cite{AG}, that the functional $I_{\gamma}$ checks condition $(H)$ and the hypotheses of the Theorem \ref{Link}. Then, there is $c_{\gamma} \in [b_{\gamma},d_{\gamma}]$ and a sequence $(u_{n}) \subset H^{1}(\mathbb{R}^{N})$ bounded (see \cite[lemma 6.9]{AG}) such that
	\begin{equation*}
		I_{\gamma}(u_{n})\rightarrow c_{\gamma}\;\;\mbox{and}\;\; \lambda_{I_{\gamma}}(u_{n})\rightarrow 0.
	\end{equation*}

	
	\begin{claim}
		There exists $\delta>0$ such that
		$$\liminf_{n} \sup_{y \in \mathbb{R}^{N}} \int_{B(y,1)} |u_{n}|^{2^{*}}\geq \delta.$$
	\end{claim}
	If the claim is not true, we must have 
	$$\liminf_{n} \sup_{y \in \mathbb{R}^{N}} \int_{B(y,1)} |u_{n}|^{2^{*}}=0.$$
	Thus, applying \cite [Lemma 2.1]{Ramos},  $u_{n}\rightarrow 0$ in $L^{2^{*}}(\mathbb{R}^{N})$ and 
	by interpolation on the Lebesgue spaces, $u_{n}\rightarrow 0$ in $L^{s}(\mathbb{R}^{N})$ for $2<s< 2^{*}$.
	On the other hand,
	\begin{equation}\label{70}
0<c=I_{\gamma}(u_{n})- \frac{1}{2}\left<w_{n}, u_{n} \right>+o_{n}(1)=\left(\frac{1}{2}- \frac{1}{2^{*}}\right) \int_{\mathbb{R}^{N}} \rho_{n} u_{n} +o_{n}(1),
	\end{equation}
	where $w_{n}= Q'(u_{n})- \rho_{n}$ with $\lambda_{I_{\gamma}}(u_{n})=||w_{n}||_{*}$ and $\rho_{n} \in  \partial \Psi_{\gamma}(u_{n})$. 
	
	Since 
	$$\int_{\mathbb{R}^{N}} \rho_{n} u_{n} \leq \gamma ||u_{n}||_q^{q}+||u_{n}||_{2^{*}}^{2^{*}}\rightarrow 0,$$
	contrary to (\ref{70}).
	
	From this, going to a subsequence if necessary, there exists $n_{0} \in \mathbb{N}$ such that
	$$\sup_{y \in \mathbb{R}^{N}} \int_{B(y,1)} |u_{n}|^{2^{*}} \geq \frac{\delta}{2}, \;n \geq n_{0}.$$
	By definition of supreme, there exists $(y_{n}) \subset \mathbb{R}^{N}$ such that
	\begin{equation*}
	\int_{B(y_{n},1)} |u_{n}|^{2^{*}} \geq \frac{\delta}{4}, \;n \geq n_{0}.
	\end{equation*}
	Then, there exists $(z_{n}) \subset \mathbb{Z}^{N}$ such that
	\begin{equation*}
	\int_{B(z_{n},1+\sqrt{N})} |u_{n}|^{2^{*}} \geq \frac{\delta}{4}, \;n \geq n_{0}.
	\end{equation*}
	Setting $v_{n}(x)=u_{n}(x+z_{n})$, we compute 
	\begin{equation}\label{84}
	\int_{B(0,1+\sqrt{N})} |v_{n}(x)|^{2^{*}}= \int_{B(z_{n},1+\sqrt{N})} |u_{n}(x)|^{2^{*}} \geq \frac{\delta}{4},\;n \geq n_{0}.
	\end{equation}
Similarly to what was done in \cite[Claim 6.12]{AG}, we have that $(v_{n}) \subset H^{1}(\mathbb{R}^{N})$ is also a $(PS)_{c_{\gamma}}$ sequence for $I_{\gamma}$.
Going to a subsequence, if necessary, let $v \in H^{1}(\mathbb{R}^{N})$ the weak limit of the sequence $(v_{n}) \subset H^{1}(\mathbb{R}^{N})$.
	
\begin{claim}\label{103}
	If $c_{\gamma}<\dfrac{S^{\frac{N}{2}}}{N}$, then $v\neq 0$. 	
\end{claim}
Suppose by contradiction $v=0$ and assume that
$$|\nabla v_{n}|^{2} \rightharpoonup \mu\;\;\mbox{and}\;\;|v_{n}|^{2^{*}} \rightharpoonup \nu \;\;\mbox{in }\;\; \mathcal{M}(\mathbb{R}^{N}).$$ 
By Concentration-Compactness Principle II due to Lions \cite{Principio}, there exist a countable  set $J$, $\nu_{j} \in \mathbb{R}_{+}$ and $(x_{j})_{j} \subset \mathbb{R}^{N}$ such that
$$\nu= \sum_{j \in J} \nu_{j} \delta_{x_{j}}$$
where $\delta_{x}$ denotes the mass of Dirac concentrated in $x \in \mathbb{R}^{N}$. In addition, 
$$\mu \geq  S \sum_{j \in J} \nu_{j}^{\frac{2}{2^{*}}} \delta_{x_{j}},$$
where 
$$S= \inf \{|\nabla u|_{2}^{2}\;;\; u \in D^{1,2}(\mathbb{R}^{N})\;,\; |u|_{2^{*}}=1\}$$
is the best Sobolev constant for immersion of $D^{1,2}(\mathbb{R}^{N})$ in $L^{2^{*}}(\mathbb{R}^{N})$. We prove that $\nu_{j}=0$ for all $j \in J$. Indeed, otherwise, suppose there is $j_{0} \in J$ such that
		\begin{equation}\label{99}
		\nu_{j_{0}}>0.
		\end{equation}
For $\delta>0$ consider the function $\varphi_{\delta} \in C_{0}^{\infty}(\mathbb{R}^{N})$
		
		$
		\varphi_{\delta}(x) = \left\{\begin{aligned}
		1,\;\; |x-x_{j_{0}}|\leq \frac{\delta}{2} \\
		0,\;\; |x-x_{j_{0}}|>\delta.
		\end{aligned}
		\right. 
		$
	
		By definition of convergence in the sense of measure theory, we get
		\begin{itemize}
			\item $\displaystyle\int_{\mathbb{R}^{N}} \varphi_{\delta} |v_{n}|^{2^*}\rightarrow  \int_{\mathbb{R}^{N}} \varphi_{\delta} d\nu$, as $n\rightarrow +\infty$.
			\item $\displaystyle \int_{\mathbb{R}^{N}} \varphi_{\delta} |\nabla v_{n}|^{2}\rightarrow \int_{\mathbb{R}^{N}} \varphi_{\delta} d\mu$, as $n\rightarrow +\infty$. 
		\end{itemize}
		Since $(v_{n})_{n}$ is bounded in $L^{2^{*}}(\mathbb{R}^{N})$, then $(|v_{n}|^{q})_{n}$ is bounded in $L^{\frac{2^{*}}{q}}(\mathbb{R}^{N})$ where $q \in (2, 2^{*})$ and $v_{n}(x)\rightarrow 0$ a.e in $\mathbb{R}^{N}$. So, 
		$$|v_{n}|^{q}\rightharpoonup 0\;\; \mbox{in}\;\; L^{\frac{2^{*}}{q}}(\mathbb{R}^{N}),$$
		that is,
		\begin{equation}\label{78}
		\int_{\mathbb{R}^{N}} |v_{n}|^{q} \varphi_{\delta} \rightarrow 0.
		\end{equation}
		By the inequality of H\"older and the limitation of $(v_{n})_{n}$ in $H^{1}(\mathbb{R}^{N})$, we get
		\begin{eqnarray}
		\left|\int_{\mathbb{R}^{N}} v_{n} \nabla v_{n} \nabla \varphi_{\delta}\right| &\leq& \left(\int_{\mathbb{R}^{N}} |v_{n}|^{2}  |\nabla \varphi_{\delta}|^{2}\right)^{\frac{1}{2}} \left(\int_{\mathbb{R}^{N}}|\nabla v_{n}|^{2}\right)^{\frac{1}{2}} \nonumber \\
		&\leq& C\left(\int_{B_{\delta}(x_{j_{0}})} |v_{n}|^{2}  |\nabla \varphi_{\delta}|^{2}\right)^{\frac{1}{2}}\nonumber.
		\end{eqnarray}
		Once $v_{n}\rightarrow 0$ in $L_{loc}^{2}(\mathbb{R}^{N})$, we obtain
		$$\left(\int_{B_{\delta}(x_{j_{0}})} |v_{n}|^{2}  |\nabla \varphi_{\delta}|^{2}\right)^{\frac{1}{2}}\rightarrow 0,\;\mbox{as}\;\; n\rightarrow \infty,$$
		showing that
		\begin{equation}\label{79}
		\int_{\mathbb{R}^{N}} v_{n} \nabla v_{n} \nabla \varphi_{\delta}\rightarrow 0, \;\mbox{as}\; n\rightarrow +\infty.
		\end{equation}
		In addition,
		\begin{equation}\label{80}
		\int_{\mathbb{R}^{N}} V(x) \varphi_{\delta} |v_{n}|^{2}\rightarrow 0, \;\mbox{as}\; n\rightarrow +\infty.
		\end{equation}
		We still have
		\begin{equation}\label{81}
		\int_{\mathbb{R}^{N}} \rho_{n} v_{n} \varphi_{\delta} \leq \gamma\int_{\mathbb{R}^{N}} \varphi_{\delta} |v_{n}|^{q} + \int_{\mathbb{R}^{N}} \varphi_{\delta} |v_{n}|^{2^{*}}.
		\end{equation}
		By the fact that $(v_{n})$ is $(PS)_{c_{\gamma}}$, there is $w_{n} \in \partial I_{\gamma}(v_{n})$ and $\rho_{n} \in \partial \Psi(v_{n})$ such that
		$$||w_{n}||=\lambda_{I_{\gamma}}(v_{n})=o_{n}(1)\;\mbox{and}\; \left<w_{n}, \phi\right>= \left<Q'(v_{n}), \phi\right>- \left<\rho_{n}, \phi\right>,\;\forall\; \phi \in  H^{1}(\mathbb{R}^{N}).$$
		By (\ref{78}), (\ref{79}), (\ref{80}) and (\ref{81})
		\begin{eqnarray}
		o_{n}(1)=\left<w_{n}, \varphi_{\delta} v_{n} \right> &=& \int_{\mathbb{R}^{N}} \nabla v_{n} \nabla (\varphi_{\delta} v_{n}) + \int_{\mathbb{R}^{N}} V(x) |v_{n}|^{2} \varphi_{\delta} - \int_{\mathbb{R}^{N}} \rho_{n} v_{n} \varphi_{\delta} \nonumber \\
		&=& \int_{\mathbb{R}^{N}} |\nabla v_{n}|^{2} \varphi_{\delta} + \int_{\mathbb{R}^{N}} v_{n} \nabla \varphi_{\delta} \nabla v_{n} +\int_{\mathbb{R}^{N}} V(x) |v_{n}|^{2} \varphi_{\delta} - \int_{\mathbb{R}^{N}} \rho_{n} v_{n} \varphi_{\delta} \nonumber \\
		&\geq& \int_{\mathbb{R}^{N}} |\nabla v_{n}|^{2} \varphi_{\delta} + \int_{\mathbb{R}^{N}} v_{n} \nabla \varphi_{\delta} \nabla v_{n} +\int_{\mathbb{R}^{N}} V(x) |v_{n}|^{2} \varphi_{\delta}+\nonumber \\
		&-& \gamma\int_{\mathbb{R}^{N}} \varphi_{\delta} |v_{n}|^{q} - \int_{\mathbb{R}^{N}} \varphi_{\delta} |v_{n}|^{2^{*}}, \nonumber
		\end{eqnarray}
		that is,
		$$0 \geq \int_{\mathbb{R}^{N}} \varphi_{\delta} d\mu - \int_{\mathbb{R}^{N}} \varphi_{\delta} d\nu,\;\forall\;\delta>0.$$
		Crossing the limit when $\delta\rightarrow 0$, by the dominated convergence theorem of Lebesgue, we get the following relationship
		$$\mu(x_{j_{0}}) \leq \nu(x_{j_{0}}),$$
		this is,
		$$S\nu_{j_{0}}^{\frac{2}{2^{*}}}\leq \nu_{j_{0}}.$$
		By (\ref{99})
		\begin{equation}\label{83}
		\nu_{j_{0}} \geq S^{\frac{N}{2}}.
		\end{equation}
		Knowing $\rho_{n}(x) \in \partial_{t} F_{\gamma}(x,v_{n}(x))$
		\begin{eqnarray}
		o_{n}(1)+c_{\gamma}&=& I_{\gamma}(v_{n})- \frac{1}{2}\left<w_{n}, v_{n} \right> \nonumber\\
		&=& \frac{1}{2}\int_{\mathbb{R}^{N}} \rho_{n} v_{n} - \int_{\mathbb{R}^{N}} F_{\gamma}(x,v_{n}) \nonumber \\
		&=& \left(\frac{1}{2}-\frac{1}{2^{*}}\right) \int_{\mathbb{R}^{N}} \rho_{n} v_{n} \geq \frac{1}{N} \int_{\mathbb{R}^{N}} |v_{n}|^{2^{*}}, \nonumber
		\end{eqnarray}
		this is,
		\begin{equation*}
		\frac{1}{N} \liminf_{n}\left(\int_{\mathbb{R}^{N}} |v_{n}|^{2^{*}} \right)\leq c_{\gamma}.
		\end{equation*}
		Since $|v_{n}|^{2^{*}}\rightharpoonup \nu$ in $\mathcal{M}^{+}(\mathbb{R}^{N})$ and (\ref{83}), we get
		\begin{eqnarray}
		c_{\gamma} &\geq& \frac{1}{N} \liminf_{n} \left(\int_{\mathbb{R}^{N}} |v_{n}|^{2^{*}} \right)\nonumber  \\
		&\geq & \frac{1}{N}\nu(\mathbb{R}^{N}) \nonumber \\
		& \geq & \frac{1}{N} \nu(\left\{x_{j_{0}}\right\}) = \frac{1}{N} \nu_{j_{0}} \geq \frac{1}{N} S^{\frac{N}{2}}>c_{\gamma}, \nonumber 
		\end{eqnarray}
		what is absurd. 
		
		Therefore, $\nu=0$ implying in $|v_{n}|^{2^{*}}\rightharpoonup 0$ in $\mathcal{M}^{+}(\mathbb{R}^{N})$ and consequently $v_{n}\rightarrow 0$ in $L_{loc}^{2^{*}}(\mathbb{R}^{N})$ contradicting (\ref{84}).
	
Now, we are ready to show the estimate from above involving the number $c_{\gamma}>0$.
\begin{itemize}
	\item Case $N\geq 4$.
\end{itemize}
\begin{remark}\label{88}
We may assume without loss of generality $V(0)<0$. So, by continuity of $V: \mathbb{R}^{N}\rightarrow \mathbb{R}$, we can choose $r>0$ such that $V(x)\leq -\beta<0$ for $x \in B_{r}$ and some $\beta>0$.
\end{remark}

Consider the function
$$\varphi_{\varepsilon}(x):= \frac{c_{N}\psi(x)\varepsilon^{\frac{N-2}{2}}}{(\varepsilon^{2}+|x|^{2})^{\frac{N-2}{2}}}$$
where $c_{N}=(N(N-2))^{\frac{N-2}{4}}$, $\varepsilon>0$ and $ \psi \in C_{0}^{\infty}(\mathbb{R}^{N})$ is such that 
$$\psi(x)=1\;\;\mbox{for}\;\; |x|\leq \frac{r}{2}\;\;\mbox{and}\;\; \psi(x)=0\;\;\mbox{for}\;\; |x| \geq r.$$ 
We shall need the following asymptotic estimates as $\varepsilon\rightarrow 0^{+}$ (see \cite{MW}).
\begin{eqnarray}
&& ||\nabla\varphi_{\varepsilon} ||_{2}^{2}=S^{\frac{N}{2}}+O(\varepsilon^{n-2}),\;\; ||\nabla\varphi_{\varepsilon} ||_{1}=O(\varepsilon^{\frac{N-2}{2}}),\;\; ||\varphi_{\varepsilon}||_{2^{*}}^{2^{*}}=S^{\frac{N}{2}}+O(\varepsilon^{N})\nonumber \\
&& \label{91} \\
&& ||\varphi_{\varepsilon}||_{2^{*}-1}^{2^{*}-1}= O(\varepsilon^{\frac{N-2}{2}}),\;\; ||\varphi_{\varepsilon}||_{1}=O(\varepsilon^{\frac{N-2}{2}})\nonumber 
\end{eqnarray}
and
\begin{eqnarray} \label{94}
||\varphi_{\varepsilon}||_{2}^{2}=\left\{\begin{array}{c}
b \varepsilon^{2}|log(\varepsilon)|+O(\varepsilon^{2}),\;\;\mbox{if}\;\; N=4 \\
b \varepsilon^{2}+ O(\varepsilon^{N-2}),\;\; \mbox{if}\;\; N\geq 5.
\end{array}
\right.
\end{eqnarray} 
where $b>0$. 

\begin{proposition}\label{96}
	Suppose $N\geq 4$, there is $\varepsilon_{0}>0$ such that for all $\varepsilon \in (0, \varepsilon_{0})$ and for all  $O(\varepsilon^{N-2})$ 
	\begin{equation*}
	\frac{\displaystyle\int_{\mathbb{R}^{N}}(|\nabla\varphi_{\varepsilon}|^{2}+V(x)\varphi_{\varepsilon}^{2})}{|\varphi_{\varepsilon}|_{2^{*}}^{2}}+O(\varepsilon^{N-2})< S.
	\end{equation*}
	
\end{proposition}
\begin{proof}
	
	In fact, given $O(\varepsilon^{N-2})$  
	\begin{eqnarray}
	\frac{\displaystyle\int_{\mathbb{R}^{N}}(|\nabla\varphi_{\varepsilon}|^{2}+V(x)\varphi_{\varepsilon}^{2})}{|\varphi_{\varepsilon}|_{2^{*}}^{2}}+O(\varepsilon^{N-2})&=& \frac{S^{\frac{N}{2}}+O(\varepsilon^{N-2})+\displaystyle\int_{\mathbb{R}^{N}} V(x)\varphi_{\varepsilon}^{2}}{[S^{\frac{N}{2}}+O(\varepsilon^{N})]^{\frac{N-2}{N}}}+O(\varepsilon^{N-2}) \nonumber \\
	&=& S\left[\frac{1+\frac{O(\varepsilon^{N-2})}{S^{\frac{N}{2}}}+\frac{1}{S^{\frac{N}{2}}}\displaystyle\int_{\mathbb{R}^{N}} V(x)\varphi_{\varepsilon}^{2}}{[1+\frac{O(\varepsilon^{N})}{S^{\frac{N}{2}}}]^{\frac{N-2}{N}}}\right]+O(\varepsilon^{N-2}) \nonumber \\
	&=&S\left[\frac{1+O(\varepsilon^{N-2})+\frac{1}{S^{\frac{N}{2}}}\displaystyle\int_{\mathbb{R}^{N}} V(x)\varphi_{\varepsilon}^{2}}{[1+O(\varepsilon^{N})]^{\frac{N-2}{N}}}\right]+O(\varepsilon^{N-2}). \nonumber
	\end{eqnarray}
	On the other hand, by (\ref{94}) and remark \ref{88} 
	\begin{eqnarray} \label{95}
	\int_{\mathbb{R}^{N}} V(x) \varphi_{\varepsilon}^{2} \leq -\beta ||\varphi_{\varepsilon}||_{2}^{2}=\left\{\begin{array}{c}
	-\beta b \varepsilon^{2}|log(\varepsilon)|+O(\varepsilon^{2})\;\;\mbox{if}\;\; N=4 \\
	- \beta b \varepsilon^{2}+ O(\varepsilon^{N-2})\;\; \mbox{if}\;\; N\geq 5.
	\end{array}
	\right.
	\end{eqnarray} 
	If $N=4$, by (\ref{95}), we obtain
	\begin{eqnarray}
	\frac{\displaystyle\int_{\mathbb{R}^{N}}(|\nabla\varphi_{\varepsilon}|^{2}+V(x)\varphi_{\varepsilon}^{2})}{|\varphi_{\varepsilon}|_{2^{*}}^{2}}+O(\varepsilon^{2}) &=& S\left[\frac{1+O(\varepsilon^{2})+\frac{1}{S^{2}}\displaystyle\int_{\mathbb{R}^{N}} V(x)\varphi_{\varepsilon}^{2}}{[1+O(\varepsilon^{4})]^{\frac{1}{2}}}\right]+O(\varepsilon^{2}) \nonumber \\
	&\leq & S\left[\frac{1+O(\varepsilon^{2})-\frac{1}{S^{2}}\beta b \varepsilon^{2}|log(\varepsilon)|+O(\varepsilon^{2})}{[1+O(\varepsilon^{4})]^{\frac{1}{2}}}\right]+O(\varepsilon^{2}) \nonumber \\
	&=& S\left(\frac{1}{[1+O(\varepsilon^{4})]^{\frac{1}{2}}}\right)+S\left(\frac{O(\varepsilon^{2})}{[1+O(\varepsilon^{4})]^{\frac{1}{2}}}\right) +\nonumber \\
	&-&S\left(\frac{\beta b \varepsilon^{2}}{S^{2}}\frac{|log(\varepsilon)|}{[1+O(\varepsilon^{4})]^{\frac{1}{2}}}\right) +O(\varepsilon^{2}). \nonumber
	\end{eqnarray}
	Note that
	\begin{itemize}
		\item [(1)]
		$\dfrac{1}{[1+O(\varepsilon^{4})]^{\frac{1}{2}}}\rightarrow 1$ as $\varepsilon \rightarrow 0^{+}$;
		\item [(2)] $\dfrac{O(\varepsilon^{2})}{[1+O(\varepsilon^{4})]^{\frac{1}{2}}}=O(\varepsilon^{2})$. 
	\end{itemize}
		Just see that
		\begin{equation*}
		\frac{\frac{O(\varepsilon^{2})}{[1+O(\varepsilon^{4})]^{\frac{1}{2}}}}{\varepsilon^{2}}= \frac{O(\varepsilon^{2})}{\varepsilon^{2}} \cdot \frac{1}{[1+O(\varepsilon^{4})]^{\frac{1}{2}}}
		\end{equation*}
		is  bounded for $\varepsilon \approx 0^{+}$.

	Follow from $(1)$,  $\varepsilon \approx 0^{+}$, that
	\begin{equation*}
	\dfrac{1}{[1+O(\varepsilon^{4})]^{\frac{1}{2}}} \geq \frac{1}{2}\Leftrightarrow \dfrac{-\beta d \varepsilon^{2}}{[1+O(\varepsilon^{4})]^{\frac{1}{2}}} \leq \frac{-\beta d \varepsilon^{2}}{2},
	\end{equation*}
	and with that
	\begin{eqnarray}
	\frac{\displaystyle\int_{\mathbb{R}^{N}}(|\nabla\varphi_{\varepsilon}|^{2}+V(x)\varphi_{\varepsilon}^{2})}{|\varphi_{\varepsilon}|_{2^{*}}^{2}}+O(\varepsilon^{2})&=& S\left(\frac{1}{[1+O(\varepsilon^{4})]^{\frac{1}{2}}}\right)+S\left(\frac{O(\varepsilon^{2})}{[1+O(\varepsilon^{4})]^{\frac{1}{2}}}\right)+\nonumber \\
	&-&S\left(\frac{\beta b \varepsilon^{2}}{S^{2}}\frac{|log(\varepsilon)|}{[1+O(\varepsilon^{4})]^{\frac{1}{2}}}\right)+O(\varepsilon^{2}) \nonumber \\
	&\leq & S\left(\frac{1}{[1+O(\varepsilon^{4})]^{\frac{1}{2}}}\right)+O(\varepsilon^{2})-S\left(\frac{\beta d \varepsilon^{2}}{2 S^{2}} |log(\varepsilon)|\right)+O(\varepsilon^{2}) \nonumber \\
	&= & S\left(\frac{1}{[1+O(\varepsilon^{4})]^{\frac{1}{2}}}\right)-S\left(\frac{\beta d \varepsilon^{2}}{2 S^{2}} |log(\varepsilon)|\right)+O(\varepsilon^{2}) \nonumber
	\end{eqnarray}
	Consider the application
	\begin{eqnarray}
	g: & [0, O(\varepsilon^{4})]& \longrightarrow \mathbb{R} \nonumber \\
	&t&\longmapsto g(t)= \frac{1}{[1+t]^{\frac{1}{2}}}. \nonumber
	\end{eqnarray}
	By the mean value theorem there is $\theta  \in (0, O(\varepsilon^{4}))$ such that
	\begin{eqnarray}
	\frac{1}{[1+O(\varepsilon^{4})]^{\frac{1}{2}}}- 1= -\frac{1}{2} [1+\theta]^{-\frac{3}{2}} O(\varepsilon^{4}), \nonumber
	\end{eqnarray}
	that is,
	\begin{equation*}
	\frac{1}{[1+O(\varepsilon^{4})]^{\frac{1}{2}}}= 1- \frac{1}{2} [1+\theta]^{-\frac{3}{2}} O(\varepsilon^{4})=1- O(\varepsilon^{4}).
	\end{equation*}
	Therefore,
	\begin{eqnarray}
	\frac{\displaystyle\int_{\mathbb{R}^{N}}(|\nabla\varphi_{\varepsilon}|^{2}+V(x)\varphi_{\varepsilon}^{2})}{|\varphi_{\varepsilon}|_{2^{*}}^{2}} +O(\varepsilon^{2})
	&\leq & S\left(\frac{1}{[1+O(\varepsilon^{4})]^{\frac{1}{2}}}\right)-S\left(\frac{\beta d \varepsilon^{2}}{2 S^{2}} |log(\varepsilon)|\right)+O(\varepsilon^{2}) \nonumber \\
	&=& S-O(\varepsilon^{4})-\frac{\beta d \varepsilon^{2}}{2 S} |log(\varepsilon)|+O(\varepsilon^{2}).\nonumber 
	\end{eqnarray}
	\begin{claim}
	There is $\varepsilon_{0}>0$ such that
	$$O(\varepsilon^{4})+O(\varepsilon^{2})-\dfrac{\beta d \varepsilon^{2}}{2 S^{2}} |log(\varepsilon)|<0,\;\forall\;\varepsilon \in (0, \varepsilon_{0}).$$ 
	\end{claim}
	Follows from the fact that
	\begin{eqnarray}
	O(\varepsilon^{4})+O(\varepsilon^{2})-\dfrac{\beta d \varepsilon^{2}}{2 S^{2}} |log(\varepsilon)|&=& \varepsilon^{2} \left[\varepsilon^{2}\frac{O(\varepsilon^{4})}{\varepsilon^{4}}+\frac{O(\varepsilon^{2})}{\varepsilon^{2}}-\dfrac{\beta d}{2 S^{2}} |log(\varepsilon)| \right] \nonumber
	\end{eqnarray}
	with the fact that:
	\begin{equation*}
	\varepsilon^{2}\dfrac{O(\varepsilon^{4})}{\varepsilon^{4}}+\dfrac{O(\varepsilon^{2})}{\varepsilon^{2}} \;\; \mbox{is bounded for}\;\; \varepsilon\approx 0^{+}\;\;\mbox{and}\;\;\lim_{\varepsilon \rightarrow 0^{+}} -\dfrac{\beta d}{2 S^{2}} |log(\varepsilon)|= -\infty.
	\end{equation*}
	Therefore, there is $\varepsilon_{0}>0$ such that
	$$\frac{\displaystyle\int_{\mathbb{R}^{N}}(|\nabla\varphi_{\varepsilon}|^{2}+V(x)\varphi_{\varepsilon}^{2})}{|\varphi_{\varepsilon}|_{2^{*}}^{2}} +O(\varepsilon^{2})<S, \;\forall\; \varepsilon \in (0, \varepsilon_{0})\;\;\mbox{and for all}\;\; O(\varepsilon^{2}).$$
	The case $N \geq 5$ is analogous.
\end{proof}

\begin{remark}\label{92}
	
	\begin{itemize}
		
		\item [(1)] Since 
		$$I_{\gamma}u)= \frac{1}{2} \int_{\mathbb{R}^{N}} (|\nabla u|^{2}+ V(x) u^{2})- \int_{\mathbb{R}^{N}} F_{\gamma}(x,u) \leq \frac{1}{2} \int_{\mathbb{R}^{N}} (|\nabla u|^{2}+ V(x) u^{2})- \frac{1}{2^{*}} \int_{\mathbb{R}^{N}} |u|^{2^{*}},$$
		defining
		$$J(u)=\frac{1}{2} \int_{\mathbb{R}^{N}} (|\nabla u|^{2}+ V(x) u^{2})- \frac{1}{2^{*}}\int_{\mathbb{R}^{N}} |u|^{2^{*}},$$ 
		we get $I_{\gamma}u) \leq J(u)$ for all $u \in H^{1}(\mathbb{R}^{N})$.
	\end{itemize}
\end{remark}
\begin{proposition}\label{97}
	For $u \in H^{1}(\mathbb{R}^{N})$, we have:
	\begin{itemize}
		\item [(i)] If $\displaystyle\int_{\mathbb{R}^{N}}[|\nabla u|^{2}+V(x)u^{2}] > 0$, then
		$$\max_{t\geq 0}J(tu)= \frac{1}{N} \left(\frac{\displaystyle\int_{\mathbb{R}^{N}}[|\nabla u|^{2}+V(x)u^{2}]}{||u||_{2^{*}}^{2}}\right)^{\frac{N}{2}}.$$
		\item [(ii)] If $\displaystyle\int_{\mathbb{R}^{N}}[|\nabla u|^{2}+V(x)u^{2}] \leq 0$, then 
		$$\max_{t\geq 0}J(tu)=0.$$
	\end{itemize}
\end{proposition}
\begin{proof}
	
	Given $u \in H^{1}(\mathbb{R}^{N})$, set the function  
	\begin{eqnarray}
	h: &[0, +\infty)&\longrightarrow \mathbb{R} \nonumber \\
	&t& \longmapsto h(t)=J(tu) \nonumber.
	\end{eqnarray}
	
	\textbf{Case (i)}.
	
	By the fact
	$$h(t)\rightarrow -\infty\;\; \mbox{as}\;\; t\rightarrow +\infty\;\; \mbox{and} \;\;h(t)>0 \;\; \mbox{for}\;\; t\approx 0,$$ 
	there is $t_{0} \in (0, +\infty)$ such that $h(t_{0})=\displaystyle\max_{t\geq 0} h(t)$. More precisely 
	$$t_{0}= \left(\frac{\displaystyle\int_{\mathbb{R}^{N}}[|\nabla u|^{2}+V(x)u^{2}]}{||u||_{2^{*}}^{2^{*}}}\right)^{\frac{1}{2^{*}-2}}.$$
	Therefore,
	$$h(t_{0})=\frac{1}{N} \left(\frac{\displaystyle\int_{\mathbb{R}^{N}}[|\nabla u|^{2}+V(x)u^{2}]}{||u||_{2^{*}}^{2}}\right)^{\frac{N}{2}}.$$
	
	\textbf{Case (ii)} 
	
	Just notice that $h(t) \leq 0$ for all $t\geq 0$ and $h(0)=0$.
\end{proof}
Before continuing we will make some considerations. First we wil need the following proposition, whose proof is in \cite[Proposition 2.2]{Chabrowski}.

\begin{proposition}\label{87}
	Suppose $V \in L^{\infty}(\mathbb{R}^{N})$ and $(V_{1})-(V_{2})$, there is $c_{0}>0$ such that
	$$||u^{-}||_{W^{1, \infty}(\mathbb{R}^{N})} \leq c_{0} ||u^{-}||_{2},\;\forall\; u^{-} \in E^{-}.$$
\end{proposition}

By the  convexity of the application $t\mapsto |t|^{2^{*}}$, H\"older inequality and Proposition \ref{87}, there is $c_{1}>0$ such that

\begin{eqnarray}
||u||^{2^{*}}_{2^{*}} &\geq& ||s\varphi_{\varepsilon}||_{2^{*}}^{2^{*}}+ 2^{*} \int_{\mathbb{R}^{N}} (s\varphi_{\varepsilon})^{2^{*}-1} u^{-} \nonumber \\
&\geq & ||s\varphi_{\varepsilon}||_{2^{*}}^{2^{*}}- c_{1}||\varphi_{\varepsilon}||^{2^{*}-1}_{2^{*}-1} ||u^{-}||_{2}\label{89}
\end{eqnarray}
and
\begin{equation}\label{90}
\int_{\mathbb{R}^{N}} (\nabla \varphi_{\varepsilon} \nabla u^{-} + V(x) \varphi_{\varepsilon} u^{-}) \leq O(\varepsilon^{\frac{N-2}{2}})||u^{-}||_{2}.
\end{equation}

\begin{proposition} \label{98}
	There is $\varepsilon_{0}>0$ such that
	$$\sup_{u \in Z_{\varepsilon}, ||u||_{2^{*}}=1} \int_{\mathbb{R}^{N}} (|\nabla u|^{2}+V(x) u^{2})<S, \;\forall\; \varepsilon \in (0, \varepsilon_{0}),$$
	where $Z_{\varepsilon}= E^{-} \oplus \mathbb{R} \varphi_{\varepsilon} \equiv E^{-} \oplus \mathbb{R} \varphi_{\varepsilon}^{+}$.
\end{proposition}
\begin{proof}
	
	Let $u= u^{-} + s \varphi_{\varepsilon}$ such that $||u^{-} + s \varphi_{\varepsilon}||_{2^{*}}=1$ . By (\ref{91}), (\ref{89}), (\ref{90}) and continuous Sobolev embeddings
	\begin{eqnarray}
	\int_{\mathbb{R}^{N}} (|\nabla u|^{2}+V(x) u^{2})&=& \int_{\mathbb{R}^{N}} (|\nabla (s\varphi_{\varepsilon})|^{2}+ V(x) (s\varphi_{\varepsilon})^{2})  + 2s\int_{\mathbb{R}^{N}} (\nabla \varphi_{\varepsilon} \nabla u^{-}+V(x) \varphi_{\varepsilon} u^{-}) + \nonumber \\
	&-& ||u^{-}||^{2}\nonumber \\
	&\leq & \frac{\displaystyle\int_{\mathbb{R}^{N}} (|\nabla \varphi_{\varepsilon}|^{2}+ V(x) \varphi_{\varepsilon}^{2})}{||\varphi_{\varepsilon}||_{2^{*}}^{2}} ||s\varphi_{\varepsilon}||_{2^{*}}^{2}+ 2s\;O(\varepsilon^{\frac{N-2}{2}})||u^{-}||_{2} - ||u^{-}||^{2} \nonumber \\
	&\leq & \frac{\displaystyle\int_{\mathbb{R}^{N}} (|\nabla \varphi_{\varepsilon}|^{2}+ V(x) \varphi_{\varepsilon}^{2})}{||\varphi_{\varepsilon}||_{2^{*}}^{2}} (1+ c_{1}||\varphi_{\varepsilon}||^{2^{*}}_{2^{*}} ||u^{-}||_{2}) +\nonumber \\
	&+&  2s\;O(\varepsilon^{\frac{N-2}{2}})||u^{-}||_{2} - \tilde{c}||u^{-}||_{2}^{2} \nonumber \\
	&\leq & \frac{\displaystyle\int_{\mathbb{R}^{N}} (|\nabla \varphi_{\varepsilon}|^{2}+ V(x) \varphi_{\varepsilon}^{2})}{||\varphi_{\varepsilon}||_{2^{*}}^{2}} + O(\varepsilon^{\frac{N-2}{2}})||u^{-}||_{2} -\tilde{c}||u^{-}||_{2}^{2} \nonumber
	\end{eqnarray}
	that is,
	$$\sup_{u \in Z_{\varepsilon}, ||u||_{2^{*}}=1} \int_{\mathbb{R}^{N}} (|\nabla u|^{2}+V(x) u^{2}) \leq \frac{\displaystyle\int_{\mathbb{R}^{N}} (|\nabla \varphi_{\varepsilon}|^{2}+ V(x) \varphi_{\varepsilon}^{2})}{||\varphi_{\varepsilon}||_{2^{*}}^{2}}+  O(\varepsilon^{\frac{N-2}{2}})||u^{-}||_{2}-\tilde{c}||u^{-}||_{2}^{2}.$$
	Knowing that 
	$$ \alpha \cdot \beta \leq \frac{\alpha^{2}}{2}+ \frac{\beta^{2}}{2},\;\forall\; \alpha, \beta \geq 0,$$
	we get
	\begin{eqnarray}
	O(\varepsilon^{\frac{N-2}{2}})||u^{-}||_{2}&=& O(\varepsilon^{\frac{N-2}{2}})(\tilde{c})^{\frac{1}{2}} ||u^{-}||_{2} \nonumber \\
	&\leq& \frac{O(\varepsilon^{\frac{N-2}{2}})^{2}}{2}+ \frac{\tilde{c} ||u^{-}||_{2}^{2}}{2} \nonumber \\
	&=& \frac{O(\varepsilon^{N-2})}{2}+ \frac{\tilde{c} ||u^{-}||_{2}^{2}}{2}. \nonumber 
	\end{eqnarray}
	\begin{eqnarray}
	\sup_{u \in Z_{\varepsilon}, ||u||_{2^{*}}=1} \int_{\mathbb{R}^{N}} (|\nabla u|^{2}+V(x) u^{2}) &\leq& \frac{\displaystyle\int_{\mathbb{R}^{N}} (|\nabla \varphi_{\varepsilon}|^{2}+ V(x) \varphi_{\varepsilon}^{2})}{||\varphi_{\varepsilon}||_{2^{*}}^{2}}+  \frac{O(\varepsilon^{N-2})}{2}+ \frac{\tilde{c} }{2}||u^{-}||_{2}^{2}-\tilde{c}||u^{-}||_{2}^{2} \nonumber \\
	&=& \frac{\displaystyle\int_{\mathbb{R}^{N}} (|\nabla \varphi_{\varepsilon}|^{2}+ V(x) \varphi_{\varepsilon}^{2})}{||\varphi_{\varepsilon}||_{2^{*}}^{2}}+  O(\varepsilon^{N-2})-\frac{\tilde{c} }{2}||u^{-}||_{2}^{2} \nonumber \\
	&\leq & \frac{\displaystyle\int_{\mathbb{R}^{N}} (|\nabla \varphi_{\varepsilon}|^{2}+ V(x) \varphi_{\varepsilon}^{2})}{||\varphi_{\varepsilon}||_{2^{*}}^{2}}+  O(\varepsilon^{N-2}). \nonumber
	\end{eqnarray}
	Therefore, by the Proposition \ref{96}, there is $\varepsilon_{0}>0$ such that
	$$\sup_{u \in Z_{\varepsilon}, ||u||_{2^{*}}=1} \int_{\mathbb{R}^{N}} (|\nabla u|^{2}+V(x) u^{2}) < S,\;\forall\; \varepsilon \in (0, \varepsilon_{0}).$$
\end{proof}

For $\varepsilon \in (0, \varepsilon_{0})$, by Proposition \ref{97} and \ref{98},  for $u \in Z_{\varepsilon}$,  we conclude
\begin{eqnarray}
J(u) &\leq& J(tu)\leq  \frac{1}{N} \left(\frac{\displaystyle\int_{\mathbb{R}^{N}}[|\nabla u|^{2}+V(x)u^{2}]}{||u||_{2^{*}}^{2}}\right)^{\frac{N}{2}} \nonumber\\
&\leq & \frac{1}{N} \left(\sup_{w \in Z_{\varepsilon}, ||w||_{2^{*}}=1}\displaystyle\int_{\mathbb{R}^{N}}[|\nabla w|^{2}+V(x)w^{2}]\right)^{\frac{N}{2}}, \nonumber
\end{eqnarray}
this is,
$$\sup_{u \in Z_{\varepsilon}} J(u) \leq \frac{1}{N} \left(\sup_{w \in Z_{\varepsilon}, ||w||_{2^{*}}=1}\displaystyle\int_{\mathbb{R}^{N}}[|\nabla w|^{2}+V(x)w^{2}]\right)^{\frac{N}{2}}< \frac{1}{N} S^\frac{N}{2}.$$
Since $c_{\gamma} \in [b_{\gamma}, d_{\gamma}]$ and 
$$d_{\gamma}= \sup_{\mathcal{M}}I_{\gamma}$$
where 
$$\mathcal{M}=\left\{u=u^{-}+t u^{+}\;; ||u||\leq \rho, t \geq 0 \;\;\mbox{and}\;\; u^{-} \in E^{-}\right\}$$
for some $u^{+} \in E^{+}\backslash \{0\}.$ We get $\mathcal{M} \subset Z_{\varepsilon}$ for $u^{+}= \varphi_{\varepsilon}^{+}$ and consequently
$$d_{\gamma} \leq \sup_{u \in Z_{\varepsilon}} J(u)<\frac{1}{N} S^\frac{N}{2}.$$
\begin{itemize}
	\item Case $N=3$.
\end{itemize}	
	
\begin{remark}\label{111}
	Consider
	 $$
	F_{1}(x,t)= \left\{\begin{aligned}
	0,\;\;\mbox{if}\;\; |t|\leq a \\
	\frac{1}{q}|t|^{q}- \frac{1}{q}a^{q},\;\; \mbox{if}\;\; |t|>a.
	\end{aligned}
	\right. 
	$$
Fixed $R>0$, for $x \in B_{R}$, we obtain: 
		
		If $|u(x)|\leq a$, $F_{1}(x,u) = 0$ and
		$$\frac{1}{q}\int_{ B_{R}}|u(x)|^{q}\leq \frac{1}{q}a^{q}|B_{R}|.$$
		
		In case $|u(x)|>a$, 
		$$\int_{ B_{R}}F_{1}(x,u)= \frac{1}{q}\int_{ B_{R}}|u(x)|^{q}-\frac{1}{q}a^{q}|B_{R}|.$$
		So,
		$$F_{1}(x,u)\geq \frac{1}{q}\int_{ B_{R}}|u(x)|^{q}-\frac{1}{q}a^{q}|B_{R}|.$$
\end{remark}
	
\begin{lemma}\label{106}
Given $z_{0} \in E^{+}\backslash \{0\}$ and $s_{0}>0$. Let $\rho>0$ given in the Lemma \ref{60}, there are $K>0$ and $R>0$ such that
	$$K||sz_{0}||_{L^{p}(B_{R})}\leq ||u^{-}+sz_{0}||_{L^{p}(B_{R})},$$
	$u^{-} \in E^{-}$, $s\geq s_{0}$, $p \in (2,2^{*})$ and $||u^{-}+sz_{0}|| \leq \rho$.
\end{lemma}
\begin{proof}
	
	Suppose that there are $s_{n} \geq s_{0}$, $u_{n}^{-} \in Y$ and $R_{n}\rightarrow +\infty$ such that
	$$ \left\| \frac{u_{n}^{-}}{s_{n}} +z_{0}\right\|_{L^{p}(B_{R_{n}})}=\frac{||u_{n}^{-}+s_{n}z_{0}||_{L^{p}(B_{R_{n}})}}{s_{n}}< \frac{||z_{0}||_{L^{p}(B_{R_{n}})}}{n},\;\forall\; n \in \mathbb{N}.$$
	So,
	\begin{equation}\label{109}
	\frac{u_{n}^{-}}{s_{n}}\longrightarrow -z_{0}\;\;\mbox{in}\;\;L^{p}(\mathbb{R}^{N}).
	\end{equation}
	On the other hand,
	$$\left\|\frac{u_{n}^{-}}{s_{n}}\right\|^{2} \leq \frac{||u_{n}^{-}+s_{n}z_{0}||^{2} }{s_{0}^{2}}\leq \frac{\rho^{2}}{s_{0}^{2}}.$$
	There is $w \in E^{-}$ such that, goingo to a subsequence if necessary,
	$$\frac{u_{n}^{-}}{s_{n}}\rightharpoonup w\;\;\mbox{in}\;\; H^{1}(\mathbb{R}^{N}),$$
	consequently
	\begin{equation}\label{110}
	\frac{u_{n}^{-}}{s_{n}}\longrightarrow w\;\; \mbox{in}\;\; L_{loc}^{p}(\mathbb{R}^{N}).
	\end{equation}
	By (\ref{109}) and (\ref{110}) we obtain
	$$w=-z_{0}\;\;\mbox{a.e in}\;\;\mathbb{R}^{N},$$
that is, $w=-z_{0} \in E^{+}\backslash \{0\}$ which contradicts the fact $w \in E^{-}$. 
	
\end{proof}

\begin{lemma}\label{112}
	Given $z_{0} \in E^{+}\backslash\{0\}$, there is $s_{0}>0$ such that
	$$d_{\gamma}=\sup_{\mathcal{M}} I_{\gamma} = \sup_{A} I_{\gamma},$$
	where 
	$$A=\{u^{-}+sz_{0}\;;\; ||u^{-}+sz_{0}||\leq \rho, u^{-} \in E^{-}\;\;\mbox{and}\;\; s\geq s_{0} \},$$
	and $\rho>0$ is given in the Lemma \ref{106}.
\end{lemma}
\begin{proof}
	
	By the definition of supreme $(s_{n}) \subset [0, +\infty)$ and $(u_{n}^{-})\subset E^{-}$ such that $||u_{n}^{-}+s_{n}z_{0}||\leq \rho$ and
	\begin{equation}\label{108}
	d_{\gamma}-\frac{1}{n} \leq I_{\gamma}(u_{n}^{-}+s_{n}z_{0})< \sup_{B_{\rho}\cap E} I_{\gamma}= d_{\gamma}>0\;,\;\forall\; n \in \mathbb{N}.
	\end{equation}
	\begin{claim}
	There is $s_{0}>0$ such that $s_{n}\geq s_{0}$ for all $n \in \mathbb{N}$.
	\end{claim}
	
	In fact, suppose that there is $(s_{n_{j}})\subset (s_{n})$ such that $s_{n_{j}}\rightarrow 0$, then
	$$I_{\gamma}(u_{n_{j}}^{-}+s_{n_{j}}z_{0})\leq \frac{s_{n_{j}}^{2}}{2}||z_{0}||^{2}\rightarrow 0.$$
	Choosing $n_{j_{0}} \in \mathbb{N}$ such that $\frac{s_{n_{j}}^{2}}{2}||z_{0}||^{2}<\frac{d_{\gamma}}{2}$, for $n_{j} \geq n_{j_0}$, we obtain
	$$d_{\gamma}-\frac{1}{n_{j}} < \frac{d_{\gamma}}{2}$$
	what contradicts (\ref{108}).
	
	So $u_{n}^{-}+s_{n}z_{0} \in A$ and
	$$d_{\gamma} \geq \sup_{A}I_{\gamma} \geq I_{\gamma} (u_{n}^{-}+s_{n}z_{0})=\sup_{B_{\rho}\cap E}I_{\gamma}+o_{n}(1)= d_{\gamma}+o_{n}(1).$$
\end{proof}

\begin{lemma} There is $\gamma>0$ such that 
	$$\sup_{u \in A}I_{\gamma}(u)< \dfrac{S^{\frac{N}{2}}}{N },$$
	where $A$ is given in the Lemma \ref{112}. In addition, for this is $\gamma>0$,  
	$$c_{\gamma} \leq d_{\gamma}< \dfrac{S^{\frac{N}{2}}}{N }.$$
\end{lemma}

\begin{proof}
	
	Since $I_{\gamma}(u) \leq J_{\gamma}(u)$ for all $u \in H^{1}(\mathbb{R}^{N})$, where
	$$J_{\gamma}(u)= \frac{1}{2}\int_{\mathbb{R}^{N}} (|\nabla u|^{2}+V(x)u^{2})- \gamma \int_{\mathbb{R}^{N}}F_{1}(x,u),$$
	so just prove the estimate to $J_{\gamma}$. 
	
	Let $u \in A$, by Remark \ref{111} and Lemma \ref{106}, there are $K>0$ and $R>0$ such that
	\begin{eqnarray}
	J_{\gamma}(u)&\leq & \frac{s^{2}}{2}||z_{0}||^{2}-\frac{1}{2}||u^{-}||^{2}- \frac{\gamma}{q} \int_{B_{R}} |u^{-}+sz_{0}|^{q}+\frac{\gamma}{q}a^{q}|B_{R}| \nonumber \\
	&\leq & \frac{s^{2}}{2}||z_{0}||^{2}- \frac{K\gamma}{q} \int_{B_{R}} |sz_{0}|^{q}+\frac{\gamma}{q}a^{q}|B_{R}|, \nonumber
	\end{eqnarray}
	that is,
	$$\sup_{u \in A}I_{\gamma}(u) \leq \sup_{s\geq s_{0}} \left(\frac{s^{2}}{2}||z_{0}||^{2}- \frac{K\gamma s^{q}}{q} \int_{B_{R}} |z_{0}|^{q}+\frac{\gamma}{q}a^{q}|B_{R}|\right).$$
	Set the function
	\begin{eqnarray}
	h:&[0, +\infty)&\longrightarrow \mathbb{R} \nonumber \\
	&s& \longmapsto h(s)=\frac{s^{2}}{2}||z_{0}||^{2}- \frac{K\gamma s^{q}}{q} ||z_{0}||_{L^{q}(B_{R})}^{q} +\frac{\gamma}{q}a^{q}|B_{R}|\nonumber.
	\end{eqnarray}
	Note that
	\begin{itemize}
		\item $h(s)\rightarrow -\infty$ as $s\rightarrow +\infty$;
		\item $h(s)>0$ for $s\approx 0^{+}$.
	\end{itemize}
	Therefore, there is $t_{0} \in (0, +\infty)$ such that $h'(t_{0})=0$, this is,
	$$t_{0}= \left(\frac{||z_{0}||^{2}}{K\gamma ||z_{0}||_{q}^{q}}\right)^{\frac{1}{q-2}}.$$
	So, 
	$$\max_{s\geq 0}h(s)=h(t_{0})=\left(\frac{1}{2}-\frac{1}{q}\right) \left(\frac{||z_{0}||^{2}}{||z_{0}||_{L^{q}(B_{R})}^{2}}\right)^{\frac{q}{q-2}} \left(\frac{1}{K\gamma}\right)^{\frac{2}{q-2}}+ \frac{\gamma}{q}a^{q}|B_{R}|.$$
	Fixed
	$$\gamma > \left(\frac{1}{2}-\frac{1}{q}\right)^{\frac{q-2}{2}} \left(\frac{||z_{0}||^{2}}{||z_{0}||_{L^{q}(B_{R})}^{2}}\right)^{\frac{q}{2}} \left(\frac{1}{K}\right)^{\frac{2}{q-2}} \left(\frac{2N}{S^{\frac{N}{2}}}\right)^{\frac{q-2}{2}},$$
	we obtain
	$$\max_{s\geq 0}h(s)=h(t_{0})< \frac{S^{\frac{N}{2}}}{2N}+ \frac{\gamma}{q}a^{q}|B_{R}|.$$
	Lastly, choosing $a \geq 0$ such that 
	$$0\leq a \leq  \left(\frac{q S^{\frac{N}{2}}}{2 \gamma N |B_{R}|}\right)^{\frac{1}{q}}$$
	we conclude that 
	$$\max_{s\geq 0}h(s)=h(t_{0})< \frac{S^{\frac{N}{2}}}{2N}+\frac{S^{\frac{N}{2}}}{2N}= \frac{S^{\frac{N}{2}}}{N} .$$
	
\end{proof}

	Now our goal is to prove that 
	$$-\Delta v(x) + V(x) v(x) \in \partial_{t} F_{\gamma}(x,v(x))\;\;\mbox{a.e in}\;\; \mathbb{R}^{N},$$
	where $v$ is the weak limit of $(v_n)$ in $H^{1}(\mathbb{R}^N)$. 
	
	From the study above, there exists $(\tilde{\omega}_{n}) \subset \partial I_{\gamma}(v_{n})$ such that $\tilde{\omega}_{n}= Q'(v_{n})- \tilde{\rho}_{n}$ and $||\tilde{\omega}_{n}||_{*}=o_{n}(1)$ where $(\tilde{\rho}_{n}) \subset \partial \Psi_{\gamma}(v_{n})$. For $\phi \in H^{1}(\mathbb{R}^{N})$, we obtain
	\begin{eqnarray}
	\left <\tilde{\rho}_{n}, \phi \right>= \left <Q'(v_{n}), \phi \right>-\left <\tilde{\omega}_{n}, \phi \right>\rightarrow \left <Q'(v), \phi \right>, \;\mbox{as}\;\;n\rightarrow +\infty,\nonumber
	\end{eqnarray}
	that is, $\tilde{\rho}_{n} \stackrel{*}{\rightharpoonup} Q'(v)$ in $(H^{1}(\mathbb{R}^{N}))^{*}$. Then, by Proposition \ref{7},  $Q'(v) \in \partial \Psi_{\gamma} (v)$. Thereby, $Q'(v)= \rho \in \partial \Psi_{\gamma} (v)$, and so, 
	$$\int_{\mathbb{R}^{N}}(\nabla v \nabla \phi + V v \phi)= \int_{\mathbb{R}^{N}} \rho \phi\;\;\mbox{for all}\;\;\phi \in H^{1}(\mathbb{R}^{N}),$$
	where $\rho(x) \in \partial_{t} F_{\gamma}(x,v(x))\;\;\mbox{a.e in}\;\; \mathbb{R}^{N}.$ Hence
	$$
	\left\{\begin{aligned}
	-\Delta v + V(x)v &= \rho(x)\;\;\mbox{in}\;\;\mathbb{R}^{N}, \\
	v \in H^{1}(\mathbb{R}^{N}).
	\end{aligned}
	\right.
	$$
	Since $\rho \in L_{loc}^{\frac{2N}{N+2}}(\mathbb{R}^{N})$, the elliptic regularity theory gives that $v \in W_{loc}^{2,\frac{2N}{N+2}}(\mathbb{R}^{N})$ and
	$$-\Delta v + V(x)v = \rho(x)\;\;\mbox{a.e in}\;\; \mathbb{R}^{N},$$
	that is,
	$$
	-\Delta v(x) + V(x) v(x) \in \partial_{t} F_{\gamma}(x,v(x))\;\;\mbox{a.e in}\;\; \mathbb{R}^{N},
	$$
	finishing the proof of Theorem \ref{Teorema1}.


\begin{thebibliography}{1} 
		
	
		
	
		\bibitem{AG} C.O Alves and G. F. Patricio,  {\it Existence of solution for a class of indefinite variatonal problems with discontinuous nonlinearity}, ArXiv:2012.03641v1[math. Ap], Preprint (2020).
		
		\bibitem{AlvesPatricio} C. O. Alves and G. F. Patricio, {\it Existence of solution for Schr\"odinger equation with discontinuous nonlinearity and asymptotically linear}, J. Math. Anal. Appl. 505, (2022), doi.org/10.1016/j.jmaa.2021.125640.
		
		\bibitem{AG1} C.O Alves and G. F. Germano,  { \it Ground state solution for a class of indefinite variational problems with critical growth}, J. Differential Equations, 265 (2018), 444-477.
		
		
		\bibitem{Alves1} C.O Alves, J.V. Gon\c calves and J.A. Santos, {\it Existence of solution for a partial differential inclusion in $\mathbb{R}^{N}$ with steep potential well}. Z. Angew. Math. Phys (2019). 
		
		\bibitem{Alves} C.O Alves, J.V. Gon\c calves and J.A. Santos, {\it Strongly nonlinear multivalued ellipti equations on a bounded domain,} J. Glob. Optim. 58 (2014), 565-593.
		
		
		\bibitem{alves1} C.O Alves, A.M. Bertone and J.V. Gon\c calves, {\it A variational approach to discontinuous problems with critical Sobolev exponents.} J. Math. Anal. App. 265 (2002), 103-127.
		
		\bibitem{alves2} C.O Alves, J.V. Gon\c calves and J.A. Santos, {\it On multiple solutions for multivalued elliptic equations under Navier boundary conditions}. J. Convex Anal. 8 (2011) 627-644
		
		\bibitem{alves3} C.O Alves and A.M. Bertone, {\it A discontinuous problem involving the p-Laplacian operator and critical exponent in $\mathbb{R}^N$ } Electron. J. Differential Equations. 2003(42), 1–10.
		
		\bibitem{alves4} C.O Alves and R.G. Nascimento, {\it Existence and concentration of solutions for a class of elliptic problem with discontinuous nonlinearity in $\mathbb{R}^{N}$}. Math. Scand. 112, 129-146. 
		
		\bibitem{Badiale} A. Ambrosetti and M. Badiale, {\it The dual variational principle and elliptic problems with discontinuous nonlinearities.}J. Math. Anal. Appl. 140, 363–373 (1989).
		
		\bibitem{Nirenberg} H. Br\'ezis and L. Nirenberg {\it Positive solutions of nonlinear elliptic equations involving critical Sobolev exponents} Comm. Pure Appl. Math, 36 (4) (1983), pp. 437-477
		
		\bibitem{Benci} V. Benci and G. Cerami,  {\it Existence of positive solutions of the equation $-\Delta u +a(x)u=u^{(N+2)/(N-2)}$ in $\mathbb{R}^{N}$}. J. Funct. Anal. 88 (1990) 90-117.
		
			\bibitem{badiale2} M. Badiale, {\it Some remarks on elliptic problems with discontinuous nonlinearities}. Rend. Sem. Mat. Univ. Politec. Torino 51, 331-342 (1993).
		
		\bibitem{cerami} G. Cerami, {\it Metodi variazionalli nello studio di problemi al contorno con parte nonlineare discontinua.} Rend. Circ. Mat.Palermo 32 (1983), 336-357 .
		
		\bibitem{Cerami} G. Cerami,  {\it Soluzioni positive di problemi con parte non lineare discontinua e applicazione a un problema di frontiera libera}. Boll. Un. Mat. Ital. B (6) 2, 321-338 (1983).
			
			\bibitem{Chabrowski} J. Chabrowski and A. Szulkin, {\it On a semilinear Schr\"odinger equation with critical Sobolev exponent}. Proc. Amer. Math. Soc. 130 (2001) 85-93.
		
		\bibitem{Clarke} F.H. Clark, {\it Generalized gradients and applications,} Trans. Amer. Math. Soc. 205 (1975), 247-262.
		
		\bibitem{Clarke1} F.H. Clark, {\it Optimization and Nonsmooth Analysis,} Wiley, New York 1983.
		
		\bibitem{Chang1} K.C. Chang, {\it Variational methods for nondifferentiable functionals and their applications to partial differential equations}. J. Math. Anal. 80, 102-129 (1981).
		
		\bibitem{chang2} K.C. Chang, {\it On the multiple solutions of the elliptic differential equations with discontinuous nonlinear terms.} Sci. Sin. 21, 139-158 (1978).
		
		\bibitem{chang3} K.C. Chang, {\it The obstacle problem and partial differential equations with discontinuous nonlinearities}. Commun. Pure Appl. Math. 33, 117-146 (1980).
		
		\bibitem{dinu} T.L. Dinu, {\it Standing wave solutions of Schr\"odinger systems with discontinuous nonlinearity in anisotropic media}. Int. J. Math. Math. Sci. 1–13 (2006)
		
		\bibitem{T.K} T.K. Donaldson and N.S. Trudinger, {\it Orlicz-Sobolev spaces and embedding theorems,} J. Funct. Anal. 8 (1971), 52-75.

		
		\bibitem{Fukagai 1} N. Fukagai, M. Ito and K. Narukawa, {\it Positive solutions of quasilinear elliptic equations with critical Orlicz-Sobolev nonlinearity on $\mathbb{R}^N$,} Funkcial. Ekvac. 49 (2006), 235-267.
		
		\bibitem{GP} L. Gasi\'nski and N. S. Papageorgiou, {\it Nonsmooth Critical Point Theory and Nonlinear Boundary Value Problems}, CHAPMAN \& HALL/CRC, 2005.
		
			\bibitem{Lins} L. Haendel F. and S. Elves A. B., {\it Quasilinear asymptotically periodic elliptic equations with critical growth}. Nonlinear Anal. 71, 2890-2905 (2009)
		
		\bibitem{Papageorgiou} N. Kourogenis and N. Papageorgiou, {\it Discontinuous quasilinear elliptic problems at resonance}. Colloq. Math., Volume 78 (1998) pp. 213-223.
		
		 \bibitem{Kryszewski} W. Kryszewski and A. Szulkin, {\it Generalized linking theorem with an application to semilinear Schr\"odinger equations,} Adv. Differential Equations (1998), 441-472.
		 
				
		 \bibitem{Principio} P. L. Lions, {\it The concentration-compactness principle in the calculus of variations. The limit case}, Rev. Mat. Iberoamericana 1 (1985), 145-201.
				
		\bibitem{Radulescu 2} P. Mironescu and V. D. R\u adulescu, {\it A multiplicity theorem for locally lipschitz periodic functionals}. J. Math. Anal. Appl., 195(3), 621–637 (1995).
				
			\bibitem{RAO} M.N. Rao and Z.D. Ren, {\it Theory of Orlicz Spaces}, Marcel Dekker, New York (1985)
				
			\bibitem{Ramos} M. Ramos, Z.-Q. Wang and M. Willem, {\it Positive Solutions for elliptic equations with critical growth in unbounded domains}, in Calculus of Variational and Differential Equations, Technion 1988, Chapam and Hall, Boca Raton, 1999.
			
			\bibitem{Radulescu} V. D. R\u adulescu, {\it Mountain pass theorems for nondifferentiable functions and applications}, Proc. Japan Acad. Ser. A Math. Sci. 69A (1993), 193-198.
			
			\bibitem{Radulescu 4}V. D. R\u adulescu, {\it Mountain pass type theorems for nondifferentiable convex functions}, Rev. Roumaine Math. Pures Appl., vol. 39, no. 1, pp. 53-62, 1994.
			
			\bibitem{Radulescu 5} V. D. R\u adulescu, {\it Locally lipschitz functionals with the strong Palais-Smale property}, Rev. Roumaine Math. Pures Appl., vol. 40, no. 3-4, pp. 355-372, 1995.
			
			\bibitem{Radulescu 1} V. D. R\u adulescu, {\it A Lusternik-Schnirelmann type theorem for locally Lipschitz functionals with applications to multivalued periodic problems}. Proc. Japan Acad. Ser. A Math. Sci., 71(7), 164-167 (1995).
			
			\bibitem{Radulescu 6} V. D.R\u adulescu, {\it Nontrivial solutions for a multivalued problem with strong resonance}, Glasg. Math. J., vol. 38, no. 1, pp. 53-59, 1996.
			
			\bibitem{Radulesco 7}  V. D. R\u adulescu, {\it Hemivariational inequalities associated to multivalued problems with strong resonance}, in Nonsmooth/Nonconvex Mechanics: Modeling, Analysis and Numerical Methods,	D. Y.Gao, R.W.Ogden, and G. E. Stavroulakis,Eds., vol. 50 of Nonconvex Optim. Appl., pp. 333-348, Kluwer Academic Publishers, Dordrecht, The Netherlands, 2001.
			
		
		\bibitem{rosario} G. Ros\'ario and S.A. Tersian, {\it An introduction to minimax theorems and their applications to Differential Equations} (2001).
		
			
		\bibitem{Stampacchia} G. Stampacchia, {\it Le problème de Dirichlet pour les équations elliptiques du second ordre à coefficients discontinues}, Ann. Inst. Fourier (Grenoble) 15 (1965), 189-288.
		
		
		
		\bibitem{Schechter} M. Schechter and W. Zou, {\it Weak linking theorems and Schr\"odinger equations with critical sobolev exponent}, ESAIM: COCV 9 601-619 (2003) DOI: 10.1051/cocv:2003029.
			
	
		
		\bibitem{MW} M. Willem, {\it Minimax Theorems}, Birkh\"auser, Boston, 1996.
		
		
		
	\end{thebibliography}
\end{document}